\documentclass[12pt]{amsart}
\usepackage[paperwidth=8.5in, paperheight=11in, margin=1in]{geometry}

\usepackage{amsmath}
\usepackage{amssymb}
\usepackage{amsthm}
\usepackage[parfill]{parskip}
\usepackage{tikz-cd} 
\usepackage{mathtools}
\usepackage{enumitem}
\usepackage[T1]{fontenc}
\usepackage{lmodern}
\usepackage{babel}
\usepackage{float}
\usepackage{dsfont}
\usepackage[numbers]{natbib}
\usepackage[unicode=true,pdfusetitle,
 bookmarks=true,bookmarksnumbered=false,bookmarksopen=false,
 breaklinks=false,pdfborder={0 0 1},backref=false,colorlinks=false]
 {hyperref}

\usepackage{enumitem}
\setlist[itemize]{leftmargin=*}
\setlist[enumerate]{leftmargin=*}

\makeatletter

\newtheorem{theorem}{Theorem}[section]
\newtheorem{lemma}{Lemma}[section]
\newtheorem{definition}{Definition}[section]
\newtheorem{proposition}{Proposition}[section]
\newtheorem{remark}{Remark}[section]

\newtheorem{example}{Example}[section]

\newlist{casenv}{enumerate}{1}
\setlist[casenv]{leftmargin=*,align=left,widest={iii}}
\setlist[casenv,1]{label=\roman{casenvi})}

\newcommand{\R}{\mathbb{R}}
\newcommand{\Z}{\mathbb{Z}}

\newcommand{\C}{\mathbb{C}}
\newcommand{\Q}{\mathbb{Q}}
\newcommand{\HH}{\mathbb{H}}

\newcommand{\OO}{\mathcal{O}}

\newcommand{\Nc}{\mathcal{N}}

\DeclareMathOperator{\Aut}{Aut}
\DeclareMathOperator{\Aff}{Aff}
\DeclareMathOperator{\Norm}{N}
\DeclareMathOperator{\Trace}{Tr}

\newcommand\blfootnote[1]{%
  \begin{NoHyper}%
  \renewcommand\thefootnote{}\footnote{#1}%
  \addtocounter{footnote}{-1}%
  \end{NoHyper}%
}

\begin{document}

\title{Automorphism groups of Inoue surfaces $S^{(+)}/S^{(-)}$}
\author{David Petcu}
\address{
		\textsc{\indent University of Bucharest, Faculty of Mathematics and Computer Science\newline 
			\indent 14 Academiei Str., Bucharest, Romania \newline
			\indent \indent and \newline
			\indent Institute of Mathematics ``Simion Stoilow'' of the Romanian Academy\newline 
			\indent 21 Calea Grivitei Street, 010702, Bucharest, Romania}}

\begin{abstract}
We present the construction of Inoue surfaces of type $S^{(+)}/S^{(-)}$ in terms of data arising from real quadratic number fields. We then describe the automorphism group of such surfaces in terms of this data.
\end{abstract}

\blfootnote{{\bf Acknowledgements.} The author wishes to express their sincere gratitude to Victor Vuletescu for their guidance and feedback throughout the course of this research. The author is also grateful to Andrei Teleman and Vicen\c tiu Pa\c sol for the insightful discussions on the subject.\\
The present work is partly supported by the PNRR-III-C9-2023-I8 grant CF 149/31.07.2023 {\em Conformal Aspects of Geometry and Dynamics}. }

\maketitle

\section{Introduction}

Inoue surfaces were introduced by Inoue as examples of minimal  non-Kahler compact complex surfaces with Kodaira dimension $-\infty$ and having no curves. Later, it was proven by Bogomolov \cite{Bog} and Teleman \cite{Tel} that they are the only surfaces  with $kod=-\infty$, first Betti number equal to $1$, second Betti number equal to $0$ and no curves.\par

\noindent
The problem of studying automorphism groups of compact complex non-K\"ahler surfaces starts with Namba's results on Hopf surfaces \cite{Nam} (1974). In \cite{FN} (2005), Fujimoto and Nakayama classify compact complex surfaces admitting nontrivial surjective endomorphisms. More recently, Prokhorov and Shramov prove in \cite{PrSh} (2020) that automorphism groups of Inoue and primary Kodaira surfaces have the Jordan property.\par

\noindent
In this paper we give a more explicit description of the automorphism group of Inoue surfaces of type $S^{(+)}/S^{(-)}$. To achieve this, we first present in Section \ref{const}, a new way of constructing the Inoue surfaces. Our approach is closely related to that of Teleman \cite{KT}, but we adopt a slightly different formulation which we find more suitable for the study of automorphisms. This is because our formulation allows us  to decide when an automorphism of the universal cover descends to an automorphism of the surface, by expressing the condition in terms of the data arising from the associated number field (see proposition \ref{bound} and proposition \ref{equ}).\par

\noindent
For a surface of type $S^{(+)}$ we first prove in Section \ref{concom} that the automorphism group is an extension of a discrete group $Q$ by $\C^*$. This result follows relatively easily from the definition of the surface, the more challenging aspect being to understand the group $Q$. We next define in Section \ref{gpq} a finite metabelian group $\mathcal{H}_{\theta, [I]}$ \ref{NTgroup} (which depends on a quadratic real number field $\Q[X]/(X^2-\theta X +1)$ and a fractional ideal $I$ of it) and show that $Q$ identifies canonically with a subgroup of it; in particular, it follows that $Q$ is finite and metabelian. We also give precise conditions for a given element of $\mathcal{H}_{\theta, [I]}$ to belong to $Q$. We use these conditions to write a short algorithm for computing $Q$. Next, in Section \ref{expls} we present some worked out examples of explicit computation of the group $Q$. Eventually, the last section \ref{s-} is devoted to the study of surfaces $S^{(-)}$. For a surface of type $S^{(-)}$, we prove that the automorphisms group is an extension of a finite metabelian group $Q$ by $\Z/2\Z$. 

\hfill

\section{Basic facts}
Inoue introduced three families of compact complex surfaces, $S^0$, $S^{(+)}$ and $S^{(-)}$. They are examples of surfaces of Kodaira class $VII$ (meaning that $b_1=1$ and $kod = -\infty$) and have the property that they do not contain curves. In this paper we will focus on the latter two types. We begin by recalling the original construction of Inoue surfaces of type $S^{(+)}$/$S^{(-)}$, as they were introduced in \cite{Ino}.\par

\noindent
{\bf Surfaces of type $\mathbf{S^{(+)}}$:} Let $N = (n_{ij}) \in \text{SL}(2,\Z)$ be a matrix with real eigenvalues $\alpha$ and $\frac{1}{\alpha}$, $\alpha > 1$. Let $(a_1,a_2)$ and $(b_1,b_2)$ be eigenvectors of $N$ associated to $\alpha$ and $\frac{1}{\alpha}$, respectively. Let $p,q,r$ be integers, $r>0$, and let $t$ be a complex number. Let $G^{(+)}_{N,p,q,r;t}$ be the group of automorphisms of $\HH\times \C$ generated by
\begin{equation}\label{Gp}
\begin{split}
&g_0(w,z) = (\alpha w,z+t)\\
&g_i(w,z) = (w + a_i, z+ b_iw+c_i)\ \ \text{ for }i\in\{1,2\}\\
&g_3(w,z) = (w, z+ d)
\end{split}
\end{equation}
where $d=\frac{b_1a_2-b_2a_1}{r}$ and $(c_1,c_2) \in \R^2$ is the unique solution of the system
\begin{equation}\label{InCon}
(N-I_2)\begin{pmatrix}
c_1\\
c_2
\end{pmatrix}=
-\begin{pmatrix}
e_1\\
e_2
\end{pmatrix}-d
\begin{pmatrix}
p\\
q
\end{pmatrix}
\end{equation}
Here, $e_1,e_2\in \R$ are defined by
$$
e_i = \frac{1}{2}n_{i1}(n_{i1}-1)a_1b_1 + \frac{1}{2}n_{i2}(n_{i2}-1)a_2b_2 + n_{i1}n_{i2}b_1a_2
$$
One may show that the action of $G_{N,p,q,r;t}^{(+)}$ on $\HH \times \C$ is properly discontinuous and without fixed points.

\begin{definition}
The compact complex surface obtained as the quotient of $\HH \times \C$ by the action of $G^{(+)}_{N,p,q,r;t}$ is denoted by $S^{(-)}_{N,p,q,r}$ and it is called an Inoue surface of type $S^{(+)}$.
\end{definition}

\noindent
{\bf Surfaces of type $\mathbf{S^{(-)}}$:} Let $N = (n_{ij})\in \text{GL}(2,\Z)$, $\det(N)=-1$, with real eigenvalues $\alpha$ and $-\frac{1}{\alpha}$, $\alpha > 1$. Let $(a_1,a_2)$ and $(b_1,b_2)$ be eigenvectors of $N$ associated to $\alpha$ and $-\frac{1}{\alpha}$, respectively. Let $p,q,r$ be integers, $r>0$. Let $G_{N,p,q,r}^{(-)}$ be the group of automorphisms of $\HH \times \C$ generated by 
\begin{equation}\label{Gm}
\begin{split}
&g_0(w,z) = (\alpha w,-z)\\
&g_i(w,z) = (w + a_i, z+ b_iw+c_i)\ \ \text{ for }i\in\{1,2\}\\
&g_3(w,z) = (w, z+ d)
\end{split}
\end{equation}
where $d=\frac{b_1a_2-b_2a_1}{r}$ and $(c_1,c_2) \in \R^2$ is the unique solution of the system
\begin{equation}\label{InCon1}
(N+I_2)\begin{pmatrix}
c_1\\
c_2
\end{pmatrix}=
-\begin{pmatrix}
e_1\\
e_2
\end{pmatrix}-d
\begin{pmatrix}
p\\
q
\end{pmatrix}
\end{equation}
Here, $e_1,e_2\in \R$ are defined by
$$
e_i = \frac{1}{2}n_{i1}(n_{i1}-1)a_1b_1 + \frac{1}{2}n_{i2}(n_{i2}-1)a_2b_2 + n_{i1}n_{i2}b_1a_2
$$
Once again, the action of $G^{(-)}_{N,p,q,r}$ on $\HH \times \C$ is properly discontinuous and without fixed points.

\begin{definition}
The compact complex surface obtained as the quotient of $\HH \times \C$ by the action of $G^{(-)}_{N,p,q,r}$ is denoted by $S^{(-)}_{N,p,q,r}$ and it is called an Inoue surface of type $S^{(-)}$.\par
\end{definition}

\noindent
As explained in the introduction of \cite{KT}, the original notation for Inoue surfaces of type $S^{(+)}$ may suggest that the biholomorphism class of the surface is independent of the choice of eigenvectors $(a_1,a_2)$ and $(b_1,b_2)$. However, this is not the case. One can show that, up to biholomorphism, varying the choice for the eigenvectors $(a_1,a_2)$ and $(b_1,b_2)$ (while keeping the matrix $N$ fixed) is equivalent to making a corresponding change in the parameters $p,q$ and $t$. The analogous statement for surfaces of type $S^{(-)}$ is also true.

\begin{theorem}[Inoue \cite{Ino}] \label{dim}
i) For $X$ Inoue surface of type $S^{(+)}$, $h^0(X, \mathcal{T}_X)=1$.\\
ii) For $X$ Inoue surface of type $S^{(-)}$, $h^0(X, \mathcal{T}_X)=0$.
\end{theorem}

\noindent
Before proceeding, we recall the following well-known result.
\begin{theorem} [Bochner, Montgomery \cite{BM}] \label{bm}
Let $X$ be a compact complex manifold, then the group of biholomorphisms $\Aut(X)$ has a natural complex Lie group structure. Moreover, the associated Lie algebra is $H^0(X, \mathcal{T}_X)$ (where $\mathcal{T}_X$ is the holomorphic tangent sheaf of $X$).
\end{theorem}

\noindent
From theorem \ref{dim} and theorem \ref{bm} above, we can see that the automorphism group of an Inoue surface is of dimension $1$ (as a complex Lie group) when the surface is of type $S^{(+)}$ and it is discrete when the surface is of type $S^{(-)}$.\par

\noindent
We will denote by $\Aff(\HH \times \C)$ the group affine maps $f: \C^2 \to \C^2$ for which $f(\HH \times \C) = \HH \times \C$. More precisely
$$\Aff(\HH \times \C)= \left\{ 
\begin{pmatrix}
w \\
z 
\end{pmatrix} \to
\begin{pmatrix}
\mu & 0 \\
\lambda & \nu
\end{pmatrix}
\begin{pmatrix}
w \\
z 
\end{pmatrix} +
\begin{pmatrix}
l \\
s 
\end{pmatrix} \vert
\mu \in \R_>;\  \lambda, s \in \C;\  \nu \in \C^*;\  l \in \R 
 \right\}$$

\begin{remark} \label{Emb}
All Inoue surfaces arise as quotients of $\mathbb{H} \times \mathbb{C}$ under the action of a subgroup of $\Aff(\mathbb{H} \times \mathbb{C})$ \cite{Ino}. For Inoue surfaces of type $S^{(+)}$ or $S^{(-)}$, this subgroup can be assumed to lie in the (real) Lie subgroup $\mathcal{G} \leq \Aff(\mathbb{H} \times \mathbb{C})$ defined as:
\[
\mathcal{G} \coloneqq 
\begin{pmatrix}
\{\pm 1\} & \mathbb{R} & \mathbb{C} \\
          & \mathbb{R}_{>} & \mathbb{R} \\
          &                 & 1
\end{pmatrix}
\]
where a matrix element of $\mathcal{G}$ acts on $(w,z)\in \HH \times \C$ by left multiplication with the column vector  $\begin{pmatrix} z & w & 1 \end{pmatrix}^\mathsf{T}$ (compare \cite{Wall}, pp. 146--147).\par
\end{remark}

\noindent
Next, we will give a different presentation of the group $\mathcal{G}$, see expression \ref{Op} below. This is precisely the presentation that is used in \cite{OeMi}, but it is written in such a way as to allow the case where $t\neq 0$. The group $\mathcal{G}$ is the semidirect product:
\[
\mathcal{G} =   \mathbb{R}^* \ltimes \mathcal{H}
\]
where $\mathcal{H}$ is the mixed Heisenberg group
  \[
  \mathcal{H} = 
  \begin{pmatrix}
  1 & \mathbb{R} & \mathbb{C} \\
    & 1          & \mathbb{R} \\
    &            & 1
  \end{pmatrix}
  \]
and $\mathbb{R}^*$ is embedded as:
  \[
  \alpha \to 
  \begin{pmatrix}
  \frac{\alpha}{|\alpha|} & 0 & 0 \\
                          & |\alpha| & 0 \\
                          &          & 1
  \end{pmatrix}
  \]
The Lie algebra of $\mathcal{H}$ is $\mathfrak{h} = \mathbb{R} \times \mathbb{R} \times \mathbb{C}$ with the bracket
$$
[(a_1, b_1, c_1), (a_2, b_2, c_2)] = (0, 0, b_1 a_2 - a_1 b_2)
$$
and exponential map
$$
\exp(a, b, c) = 
\begin{pmatrix}
1 & b & c + \frac{ab}{2} \\
  & 1 & a \\
  &   & 1
\end{pmatrix}
$$
In this case the exponential map is a diffeomorphism, which allows us to identify $\mathcal{H}$ with its Lie algebra, $\mathfrak{h}$. By this we mean that the elements of $\mathcal{H}$ can be expressed in logarithmic coordinates via the inverse of the exponential map. With respect to these coordinates, the group law is given by the Baker-Campbell-Hausdorff formula:
$$
(a_1, b_1, c_1) \cdot (a_2, b_2, c_2) = \left(a_1 + a_2,\ b_1 + b_2,\ c_1 + c_2 - \frac{a_1 b_2 - b_1 a_2}{2}\right)
$$
Thus, we see that $\mathcal{G} \simeq \mathbb{R}^* \ltimes \mathfrak{h}$ and the group operation may be explicitly written as:
\begin{equation} \label{Op}
\begin{split}
&[\alpha_1, (a_1, b_1, c_1)] \cdot [\alpha_2, (a_2, b_2, c_2)]= \\
& = \left[\alpha_1 \alpha_2,\ \left(a_1 + |\alpha_1| a_2,\ b_1 + \frac{\alpha_1}{|\alpha_1|^2} b_2,\ c_1 + \frac{\alpha_1}{|\alpha_1|} c_2 - \frac{a_1 b_2 \frac{\alpha_1}{|\alpha_1|^2} - b_1 a_2 |\alpha_1|}{2} \right)\right]
\end{split}
\end{equation}
This is our preferred way of presenting the group $\mathcal{G}$ and it will be used throughout the following sections. The next result will prove very useful, since it reduces our search for the automorphism groups to algebraic computations in the group of affinities.\par

\begin{theorem}[Khaled, Teleman \cite{KT}] \label{af}
Let $S$ be an Inoue surface, $f \in \Aut(S)$ and $\tilde{f}:\HH\times\C \to \HH\times\C$ a lift of $f$. Then $\tilde{f} \in \Aff(\HH \times \C)$.
\end{theorem}

\section{Surfaces of type $S^{(+)}$}
\subsection{A slightly different construction} \label{const}
In this section we will introduce an alternative way of presenting Inoue surfaces $S^{(+)}$. In definition \ref{gr}, we introduce a group of analytic automorphisms of $\HH \times \C$, depending on a set of parameters arising from a real quadratic number field. Then, in definition \ref{quo}, we introduce the associated surface of such a group. In proposition \ref{p1} we show that the resulting surfaces are indeed of Inoue type $S^{(+)}$ and in proposition \ref{p2}, we prove that our construction recovers all the possible examples of surfaces of this type. We begin by setting up some notations and making some remarks on the properties of the data that will be used for the construction.\par

Let $\theta$ be an integer, $\theta \geq 3$, consider the polynomial 
\begin{equation} \label{poly}
f_\theta(X)=X^2-\theta X+1 
\end{equation}
and the associated real quadratic number field $K=\Q[X]/(f_{\theta})$.
Let $u_\theta \coloneqq \hat{X} \in K$, it is clear that $u_\theta$ is an algebraic integer and, in fact, a unit in $\OO_K$. Let $\sigma_1, \sigma_2 : K \to \R$ be the two embeddings of $K$. We now fix the convention that $\sigma_1(u_\theta)>1$. The set 
\begin{equation}\label{halfpos}
\OO_K^{\times, >} \coloneqq \{u \in \OO_K^\times \vert \sigma_1(u) > 0\}
\end{equation}
is a subgroup of the multiplicative group of units $\OO_K^\times$.
Since $K$ is real quadratic, from Dirichlet's unit theorem, we get that $\OO_K^\times \simeq \{\pm 1\} \times \Z$. Let $\eta$ be the fundamental unit of $\OO_K$ ($\sigma_1(\eta)>1$), we then have that $\OO_K^{\times, >} \simeq \langle \eta \rangle$ and $\langle u_\theta \rangle \leq \OO_K^{\times, >}$ is a subgroup of finite index. More precisely 
\begin{equation}\label{index}
[\OO_K^{\times,>}: \langle u_\theta\rangle] = \frac{\text{log}(\sigma_1(u_\theta))}{\text{log}(\sigma_1(\eta))}
\end{equation}
Let $I$  be a fractional ideal of the order $\Z[u_\theta]$; then, as a group, $I$ is isomorphic to $\Z^2$. Hence we may choose some $\Z$-basis $(x_1,x_2)$ for $I$. The multiplicative action of $u_\theta$ on the ideal $I$, with respect to the basis $(x_1,x_2)$, is expressed by a matrix $N_{\theta,x_1,x_2}\in \text{SL}(2,\Z)$ with trace $\theta$, such that
\begin{equation}\label{AsMat}
N_{\theta,x_1,x_2} \begin{pmatrix}x_1 \\x_2\end{pmatrix}=u_0\begin{pmatrix}x_1 \\x_2\end{pmatrix}
\end{equation}

\begin{definition}\label{msub}
Let $K$ be the real quadratic number field $K = \Q[X]/ (X^2-\theta X +1)$ for some integer $\theta \geq 3$. Let $u_\theta = \hat{X} \in K$ and let $I$ be a fractional ideal of $\Z[u_\theta]$. Denote by $ ^I[\OO_K^{\times,>}]$ the subgroup of $\OO_K^{\times,>}$ (see \ref{halfpos}) containing those units that leave the ideal $I$ invariant (as a set), that is:
$$^I[\OO_K^{\times,>}]\coloneqq \{u\in \OO_K^{\times,>}\vert uI=I\}$$
The condition that a unit $u$ is contained in $^I[\OO_K^{\times,>}]$ is equivalent to $I$ being a $\Z[u]$-fractional ideal.
\end{definition}
\noindent
We will use $\Norm (\cdot)$ and $\Trace(\cdot)$ for the norm and trace in $K$ respectively and we will denote by $\varphi : K \to K $  the nontrivial Galois automorphism of $K$. We will use $\chi$ to denote the $\Q$-bilinear form $\chi: K \times K \to \R$ given by 
\begin{equation}\label{chi}
\chi(x,y)= \sigma_1(x)\sigma_2(y)-\sigma_1(y)\sigma_2(x)
\end{equation}
\noindent
Although the following properties of the form $\chi$ are straightforward to verify, we record them here, as they will be intensively used in the upcoming computations:
\begin{itemize}
\item[$\bullet$] The form $\chi$ is antisymmetric.
\item[$\bullet$] $\chi(xy_1, y_2)= \chi(y_1, \varphi(x) y_2)$, for all $x,y_1,y_2 \in K$.
\item[$\bullet$] $\chi(xy_1, xy_2)= \Norm(x)\chi(y_1,y_2)$, for all $x,y_1,y_2 \in K.$
\end{itemize}
\begin{proposition}\label{emb1}
There is a natural map from the set $\OO_K^{\times, >} \times K \times \C$ to the group $\mathcal{G}$ (see \ref{Op}):
$$\iota_{\theta}^{(+)}: \OO_K^{\times, >} \times K \times \C \to \mathcal{G}\ \ \ \ \ \ \iota_{\theta}^{(+)}([u,x,t])=[\sigma_1(u)\Norm(u), (\sigma_1(x), \sigma_2(x),t)]$$
The map $\iota_{\theta}^{(+)}$ is injective and its image is a subgroup of $\mathcal{G}$.
\end{proposition}

\begin{proof}
To prove that the image is a subgroup it is enough to observe that:
\begin{casenv}

\item $\iota_{\theta}^{(+)}([u,x,t])\iota_{\theta}^{(+)}([v,y,s])=\iota_{\theta}^{(+)}\left(\left[uv,\ x+uy,\ t+\Norm(u)s - \frac{\chi(x,uy)}{2}\right]\right)$

\item $\iota_{\theta}^{(+)}([1,0,0])=[1,(0,0,0)]$,  which is the identity of $\mathcal{G}$.

\item $\iota_{\theta}^{(+)}([u,x,t])\iota_{\theta}^{(+)}\left(\left[\frac{1}{u}, -\frac{x}{u}, -\frac{t}{\Norm(u)}\right]\right)=[1,(0,0,0)]$
\end{casenv}
The injectivity follows directly from the definition.
\end{proof}

\begin{definition}\label{Incl}
Let $\mathcal{G}_{\theta}^{(+)}$ be the group $\OO_K^{\times, >} \times K \times \C$, with the group structure induced by the above embedding $\iota_{\theta}^{(+)}$ in $\mathcal{G}$.
\end{definition}

\noindent
As seen in the proof of proposition \ref{emb1}, the group structure on $\mathcal{G}_\theta^{(+)}$ is given by 
\begin{equation*}[u,x,t][v,y,s]=\left[uv,\ x+uy,\ t+\Norm(u)s - \frac{\chi(x,uy)}{2}\right]
\end{equation*}
This implies that $\mathcal{G}_\theta^{(+)}$ is a group extension of the form
$$\begin{tikzcd}[ampersand replacement=\&,cramped]
	1 \& \C \& {\mathcal{G}_{\theta}^{(+)}} \& {\OO_K^{\times,>}\ltimes K} \& 1
	\arrow[from=1-1, to=1-2]
	\arrow[from=1-2, to=1-3]
	\arrow[from=1-3, to=1-4]
	\arrow[from=1-4, to=1-5]
\end{tikzcd}$$
where $\OO_K^{\times,>}$ acts on $K$ by multiplication. Moreover, the subgroup of $\mathcal{G}_{\theta}^{(+)}$ coming from the inclusion of the left term is the center of $\mathcal{G}_{\theta}^{(+)}$. Since the group $\mathcal{G}$ embeds in $\Aff(\HH \times \C) $ (remark \ref{Emb}), so does $\mathcal{G}_\theta^{(+)}$. Explicitly writing the left action of $\mathcal{G}_\theta^{(+)}$ on the complex domain $\HH\times \C$, we get
$$[u,x,t](w,z)=\left(\sigma_1(u)w + \sigma_1(x),\ \Norm(u)z + \sigma_2(x)\sigma_1(u)w + \frac{\Norm(x)}{2} + t\right)$$

\begin{definition}\label{gr}
Let $\theta$ be an integer, $\theta \geq 3$, take $f_\theta$ to be the polynomial defined by \ref{poly}. Let $K=\Q[X]/(f_\theta)$ and $u_\theta \coloneq \hat{X} \in K$. Let $r$ be a positive integer and $I = {}_\Z\langle x_1,x_2\rangle$, a fractional ideal of the order $\Z[u_\theta]$. Let $e$ be an element in $K$ and $t$ a complex number. We will denote by $\Gamma_{\theta,r,x_1,x_2,e;t}^{(+)}$ the subgroup of $\mathcal{G}_\theta^{(+)}$ (\ref{Incl}) generated by the elements:
$$g_0=[u_\theta,t];\ \ g_i=[1,x_i,\chi(x_i, e)], i \in \{1,2\};\ \ g_3=\left[1,0,-\frac{\chi(x_1,x_2)}{r}\right]$$
We will say that the group $\Gamma_{\theta,r,x_1,x_2,e;t}^{(+)}$ is in standard form if its center is generated by $g_3$.\par
\end{definition}

\noindent
The following lemma gives an equivalent statement to the standard form condition. This formulation, in terms of the parameter $e$, mirrors the first compatibility condition of definition 2.13 in \cite{KT}. 

\begin{lemma} \label{sf}
 Let $\Gamma =\Gamma_{\theta,r,x_2,x_2,e;t}^{(+)}$ and $N = (n_{ij})\coloneq N_{\theta, x_1,x_2}$, the associated matrix (\ref{AsMat}). Then, the standard form condition for $\Gamma$ is equivalent to 
\begin{equation}\label{sfl}
\frac{1-u_\theta}{u_\theta}e + \frac{n_{21}n_{22}}{2}x_1-\frac{n_{11}n_{12}}{2}x_2 \in Ir^{-1}
\end{equation}
\end{lemma}
\begin{proof}
We begin by observing that the center of $\Gamma$ is generated by $g_3$ and $g_1^{n_{i1}}g_2^{n_{i2}}(g_0 g_ig_0^{-1})^{-1}$, $i\in \{1,2\}$. Hence the standard form condition is equivalent to $g_1^{n_{i1}}g_2^{n_{i2}}(g_0 g_ig_0^{-1})^{-1} \in \langle g_3 \rangle$, $i \in \{1,2\}$. In other words, $\Gamma$ is in standard form if and only if there exist integers $p$ and $q$, such that $g_0g_1g_0^{-1}= g_1^{n_{11}}g_2^{n_{12}}g_3^{p}$ and $g_0g_2g_0^{-1}= g_1^{n_{21}}g_2^{n_{22}}g_3^{q}$. By direct computation, we see that the unique choice of $e \in K$ for which the above relations are satisfied is 
\begin{equation}\label{ee}
e=\frac{u_\theta}{1-u_\theta}\left(\left(\frac{n_{11}n_{12}}{2}+\frac{p}{r}\right)x_2-\left(\frac{n_{21}n_{22}}{2}+\frac{q}{r}\right)x_1\right)
\end{equation} 
Thus, the proof is complete.
\end{proof}

\begin{proposition} \label{p1}
Let $\Gamma$ be a group of type $\Gamma_{\theta, r,x_1,x_2,e;t}^{(+)}$, as defined by \ref{gr}. Assume that $\Gamma$ is in standard form, then, when seen as a subgroup of $\Aff(\HH \times \C)$, $\Gamma$ is a group of Inoue type $G^{(+)}$ (\ref{Gp}).
\end{proposition}

\begin{proof}
Writing explicitly the action of the generators of $\Gamma$ on $\HH \times \C$, we get
\begin{align*}
&g_0(w,z)=(\sigma_1(u_\theta)w,\ z+t)\\
&g_i(w,z)=(w+\sigma_1(x_i),\ z+\sigma_2(x_i)w+\frac{\Norm(x_i)}{2} + \chi(x_i,e))\\
&g_3(w,z)=(w,z-\frac{\chi(x_1,x_2)}{r})
\end{align*}
Let $N=(n_{ij})$ be the associated matrix of $u_\theta$ (\ref{AsMat}). 
$$N \begin{pmatrix}x_1 \\x_2\end{pmatrix}=u_\theta\begin{pmatrix}x_1 \\x_2\end{pmatrix} \implies N\begin{pmatrix}\sigma_i(x_1) \\\sigma_i(x_2)\end{pmatrix} = \sigma_i(u_\theta)\begin{pmatrix}\sigma_i(x_1) \\\sigma_i(x_2)\end{pmatrix}\text{, for } i \in \{1,2\}$$
In order to match Inoue's original notation, we set $\alpha = \sigma_1(u_\theta)$, $a_i = \sigma_1(x_i)$, $b_i=\sigma_2(x_i)$ and $c_i=\frac{\Norm(x_i)}{2} + \chi(x_i,e)$. We have seen that $(a_i)$ and $(b_i)$ are eigenvectors of $N$ corresponding to the eigenvalues $\alpha$ and $\frac{1}{\alpha}$, respectively. We only have to check that there exist integers $p,q$ such that equality (\ref{InCon}) is satisfied. By simple computations, we see that this is equivalent to 
$$\begin{pmatrix} \chi((u_\theta -1)x_1,e) \\
\chi((u_\theta -1)x_2,e)\end{pmatrix} - \begin{pmatrix}n_{11}n_{12}\\
n_{21}n_{22}
\end{pmatrix} \frac{\chi(x_1,x_2)}{2} = \frac{\chi(x_1,x_2)}{r}\begin{pmatrix}p \\q \end{pmatrix}$$
so, by lemma \ref{sf},  the existence of $p$ and $q$ is equivalent to the fact that $\Gamma_{\theta,r,x_1,x_2,e;t}^{(+)}$ is in standard form. This finishes the proof.
\end{proof}

\begin{definition}\label{quo}
We will denote by $S^{(+)}_{\theta,r,x_1,x_2,e;t}$ the Inoue surface obtained as the quotient of $\HH \times \C$ by the left action of $\Gamma_{\theta,r,x_1,x_2,e;t}^{(+)}$.
\end{definition}

\begin{proposition} \label{p2}
Any Inoue surface of type $S^{(+)}$ is biholomorphic to one of the form $S^{(+)}_{\theta,r,x_1,x_2,e;t}$.
\end{proposition}

\begin{proof}
Let $N \in \text{SL}(2,\Z)$ with real eigenvalues $\alpha, \frac{1}{\alpha}$, $\alpha >1$. Let $p,q,r$ be integers, $r>0$, and let $t\in \C$. Clearly, $\alpha$ is the root of the polynomial $X^2-\theta X+1$, where $\theta$ is the trace of $N$. Thus, if we let $K\coloneq \Q[X]/(X^2 -\theta X + 1)$ and $u_\theta = \hat{X} \in K$, then $\sigma_1(u_\theta) = \alpha$. Next, we observe that it is always possible to choose the two eigenvectors $(a_1,a_2)$ and $(b_1,b_2)$ such that $(a_1,a_2), (b_1,b_2) \in \Q(\alpha)^2$ and having the property that $a_i$ and $b_i$ are conjugate in $\Q(\alpha)$, for $i\in \{1,2\}$. This implies that there exist $x_1,x_2 \in K$ such that $\sigma_1(x_i)=a_i$, $\sigma_2(x_i)=b_i$, for $i\in\{1,2\}$. Then, the subgroup $_\Z\langle x_1,x_2 \rangle \leq K$ is a fractional ideal of the order $\Z[u_\theta]$. Let $e\in K$ be defined by expression \ref{ee}. Then, we have that the Inoue surface $S^{(+)}_{N,p,q,r;t}$ (with eiegenvectors $(a_i)$ and $(b_i)$ chosen as above) is biholomorphic to $S^{(+)}_{\theta,r,x_1,x_2,e;t}$.
\end{proof}

\noindent
For the remainder of this section, we will analyze the symmetries of the parameters 
$$(\theta, r, x_1, x_2, e; t)$$
up to biholomorphism (or deformation equivalence) of the Inoue surface they define (Lemma \ref{Sym}). The statements of the Lemma below are equivalent to the results in \cite{KT} so we will not prove them. This is also true for the final proposition of the section, which is the main classification result of \cite{KT} rephrased in terms of our construction.

\begin{lemma}\label{Sym}
Let $S$ be the Inoue surface $S_{\theta, r, x_1,x_2,e;t}^{(+)}$. Using the same notation as in definition \ref{gr}, let $K = \Q[X]/(X^2-\theta X + 1)$; $u_\theta = \hat{X} \in K$ and $I = {}_\Z\langle x_1, x_2\rangle$.
\begin{casenv}
\item If the surfaces $S$ and $S_{\theta',r',x_1',x_2',e';t'}^{(+)}$ are biholomorphic, then $\theta = \theta'$ and $r=r'$. In other words, the pair $(\theta,r)$ is a biholomorphism invariant.
\item If the surfaces $S$ and $S_{\theta,r,x_1',x_2',e';t'}^{(+)}$ are biholomorphic, then the fractional ideals $_\Z\langle x_1,x_2\rangle$ and $_\Z\langle x_1',x_2'\rangle$ represent the same ideal class in $\text{Cl}(\Z[u_\theta])$.
\item If $(y_1,y_2)$ is another $\Z$-basis for the fractional ideal $I={}_\Z\langle x_1,x_2\rangle$, then there exists some $e' \in K$ such that $S \simeq S_{\theta,r,y_1,y_2,e';t}^{(+)}$.
\item Let $y$ be a nonzero element of the field $K$. If $\sigma_1(y)>0$, then $S \simeq S_{\theta,r, yx_1,yx_2,ye;\Norm(y)t}^{(+)}$. Similarly, if $\sigma_1(y)<0$, then $S \simeq S_{\theta,r, yx_1,yx_2,-ye;\Norm(y)t}^{(+)}$. Combining the statements iii) and iv), we see that if $I={}_\Z\langle x_1,x_2\rangle$ and  $J={}_\Z\langle y_1,y_2\rangle$ represent the same ideal class in $\text{Cl}(\Z[u_\theta])$, then there exist some $e' \in K$ and $t' \in \C$, such that $S \simeq S_{\theta,r,y_1,y_2,e';t'}^{(+)}$.
\item Let $e' \in K$ with $e-e' \in Ir^{-1}$, then $S \simeq S^{(+)}_{\theta, r , x_1, x_2, e';t}$.
\item Let $e' \in K$ with $e-e' \in I(1-u_\theta)^{-1}$, then there exists some $t' \in \C$ such that $S \simeq S^{(+)}_{\theta, r, x_1, x_2, e';t'}$. Combining statements v) and vi), we see that changing just the parameter $e$ by an element in $Ir^{-1} + I(1-u_\theta)^{-1}$ does not change the deformation class of the surface.
\item If $r$ is even, then for any $u\in {}^I[\OO_K^{\times,>}]$, there exists $t' \in \C$, such that $S \simeq S^{(+)}_{\theta,r,x_1,x_2,ue;t}$. If $r$ is odd, then $e$ may be decomposed as $e=e_1 + e_2$, where $e_1 \in Ir^{-1}(1-u_\theta)^{-1}$ and $e_2 \in I2^{-1}(1-u_\theta)^{-1}$. Let $e' = ue_1 + e_2$, then there exist $t'$, such that $S \simeq S^{(+)}_{\theta, r, x_1,x_2,e';t'}$.
\end{casenv}
\end{lemma}

\begin{proposition}
Fix $\theta$ and $r$, then for any ideal class $[I]$ in $\text{Cl}(\Z[u_\theta])$, the set of deformation classes of Inoue surfaces of type $S^{(+)}_{\theta,r,x_1,x_2,e;t}$ having $_\Z\langle x_1,x_2\rangle \in [I]$ is in bijection to $I/I(r, 1-u_\theta)$ modulo the multiplicative action of ${}^I[\OO_K^{\times,>}]$.
\end{proposition}

\subsection{The connected component of the identity}\label{concom} We are now ready to study the groups of biholomorphisms of Inoue surfaces of type $S^{(+)}$. In this section we will show that the connected component of the identity of the automorphism group of such surface is isomorphic to $\C^*$.\par

\noindent
Let $\theta \geq 3$ and $r > 0$ be integers and $t \in \C$. Let $K=\Q[X]/(X^2 - \theta X + 1)$ and $u_\theta \coloneq \hat{X} \in K$. Let $I = {}_\Z\langle x_1,x_2\rangle$ be a fractional ideal of $\Z[u_\theta]$ and let $e \in K$ . 
Let $X = S^{(+)}_{\theta, r, x_1, x_2, e ; t}$ and $\Gamma = \Gamma^{(+)}_{\theta,r,x_1,x_2,e;t}$; we know that $\Gamma \leq \Aff(\HH \times \C)$ and, from the theorem \ref{af}, we see that 
$$\Aut(X) \simeq \Nc(\Gamma)/ \Gamma$$
where $\Nc(\Gamma)$ is the normalizer of $\Gamma$ in $\Aff(\HH \times \C)$.

\begin{proposition}
The connected component of the identity of the complex Lie group $\Aut(X)$ is isomorphic to $\C^*$.
\end{proposition}

\begin{proof}
Let $T_\C$ be the subgroup of translations of $\HH \times \C$ in the $z$-direction, that is:
\begin{equation} \label{trans}
T_\C = 
\left\{
\begin{pmatrix}
w \\
z 
\end{pmatrix} \to
\begin{pmatrix}
w \\
z 
\end{pmatrix} +
\begin{pmatrix}
0 \\
s 
\end{pmatrix} \vert s\in \C
 \right\}
\end{equation} \label{trans}
It is easy to see that $T_\C \unlhd \Nc(\Gamma)$, hence we have the following diagram of normal subgroups:
\[\begin{tikzcd}
	& {\Nc(\Gamma)} \\
	& {\Gamma\cdot T_\C} \\
	\Gamma && {T_\C} \\
	& {\Gamma\cap T_\C}
	\arrow[no head, from=1-2, to=2-2]
	\arrow[no head, from=2-2, to=3-3]
	\arrow[no head, from=3-1, to=2-2]
	\arrow[no head, from=4-2, to=3-1]
	\arrow[no head, from=4-2, to=3-3]
\end{tikzcd}\]
Since $$\frac{\Nc(\Gamma)}{(\Gamma \cdot T_\C)} \simeq \frac{\Nc(\Gamma)/\Gamma}{(\Gamma \cdot T_\C)/\Gamma}$$ there is a short exact sequence of groups:
\[\begin{tikzcd}
	1 & {\frac{\Gamma \cdot T_\C}{\Gamma}} & {\Aut(X)} & {\frac{\Nc(\Gamma)}{\Gamma\cdot T_\C}} & 1
	\arrow[from=1-1, to=1-2]
	\arrow[from=1-2, to=1-3]
	\arrow[from=1-3, to=1-4]
	\arrow[from=1-4, to=1-5]
\end{tikzcd}\]
We observe that $$\frac{\Gamma \cdot T_\C}{\Gamma} \simeq \frac{T_\C}{\Gamma \cap T_\C} \simeq \frac{\C}{d\Z}\simeq \C^*$$
where $d = \frac{\chi(x_1,x_2)}{r}$, also we have that $$\frac{\Nc(\Gamma)}{\Gamma \cdot T_\C} \simeq \frac{\Nc(\Gamma)/ T_\C}{(\Gamma \cdot T_\C)/ T_\C} \simeq \frac{\Nc(\Gamma)/ T_\C}{\Gamma / (\Gamma \cap T_\C)}$$
Denote this last quotient by $Q$. The short exact sequence becomes:
\[\begin{tikzcd}
	1 & {\C^*} & {\Aut(X)} & Q & 1
	\arrow[from=1-1, to=1-2]
	\arrow[from=1-2, to=1-3]
	\arrow[from=1-3, to=1-4]
	\arrow[from=1-4, to=1-5]
\end{tikzcd}\]
This is a short exact sequence of complex Lie groups and, since $\Aut(X)$ is of dimension $1$, $\C^*$ is the connected component of the identity and $Q$ is discrete. In fact, as we shall see in the following section, it is finite. 
\end{proof}

\subsection{The connected components of $\Aut(X)$ (the group $Q$)}\label{gpq}
As before, let $\Gamma = \Gamma^{(+)}_{\theta,r,x_1,x_2,e;t}$
be the group defined in \ref{gr}, and let $X = S^{(+)}_{\theta,r,x_1,x_2,e;t}$ denote the corresponding Inoue surface. In this section, we develop an algorithm for determining the group of connected components of $\Aut(X)$, previously denoted by $Q$. The strategy is as follows: we begin by introducing a finite metabelian group $\mathcal{H}_{\theta,[I]}$, which is constructed from the number-theoretic parameters defining $X$. We then prove that $Q$ embeds as a subgroup of $\mathcal{H}_{\theta,[I]}$, thereby showing that $Q$ is itself finite and metabelian. Next, we give two easily verifiable conditions that characterize when a given element of $\mathcal{H}_{\theta,[I]}$ belongs to $Q$. Finally, we collect these steps into an algorithm: starting with the parameters $(\theta,r,x_1,x_2,e;t)$, one constructs $\mathcal{H}_{\theta,[I]}$, applies the membership conditions, and thus obtains the desired description of $Q$.

\begin{definition}\label{NTgroup}
Using the same notations as in definition \ref{msub} above, we define the following group:
$$\mathcal{H}_{\theta, [I]}\coloneq \frac{{}^I[\OO_K^{\times,>}]\ltimes I (1-u_\theta)^{-1}}{\langle u_\theta \rangle \ltimes I} \simeq\frac{{}^I[\OO_K^{\times,>}]}{\langle u_\theta \rangle} \ltimes \frac{I(1-u_\theta)^{-1}}{I}$$
where ${}^I[\OO_K^{\times,>}]$ acts on $I (1-u_\theta)^{-1}$ by multiplication.
\end{definition}

\noindent
It is easy to see that, up to isomorphism, the group $\mathcal{H}_{\theta, [I]}$ depends only on the class of $I$ in the ideal class monoid $\text{Cl}(\Z[u_\theta])$. 

\begin{proposition} \label{bound} Let $Q$ denote the group of connected components of the automorphism group of the surface $X$. Then, $Q$ embeds as a subgroup of $\mathcal{H}_{\theta, [I]}$ (definition \ref{NTgroup}). In particular, it follows that $Q$ is finite and metabelian.
\end{proposition}

\begin{proof}
In the previous section we have seen that 
$$Q \simeq \frac{\Nc(\Gamma)/T_\C}{\Gamma/(\Gamma \cap T_\C)}$$
where $\Nc(\Gamma)$ denotes the normalizer of $\Gamma $ in $\Aff(\HH \times \C)$ and $T_\C$ is defined by \ref{trans}. It is straightforward to verify that $\Gamma \cap T_\C = \langle g_3 \rangle $ and hence
\begin{equation}\label{deno}
\Gamma/ (\Gamma \cap T_\C) \simeq \langle u_\theta\rangle \ltimes I
\end{equation}
In the upcoming lemma (\ref{norm}) we prove that $\Nc(\Gamma)$ is a subgroup of 
$$\left\{[\mu,l,s]\in \mathcal{G}^{(+)}_\theta\ \vert\ \mu \in {}^I[\OO_K^{\times, >}],\ l \in I(1-u_\theta)^{-1},\ s\in \C\right\}$$
It follows that 
\begin{equation}\label{nume}
\Nc(\Gamma)/T_\C \leq {}^I[\OO_K^{\times, >}] \ltimes I(1-u_\theta)^{-1}
\end{equation}
From (\ref{deno}) and $\ref{nume}$, we get that $ Q \leq \mathcal{H}_{\theta,[I]}$, which is the desired conclusion.
\end{proof}

\begin{lemma}\label{norm} 
Let  $\Nc(\Gamma)$ be the normalizer of $\Gamma$ in $\Aff(\HH \times \C)$. Then, $\Nc(\Gamma)$ is contained in $\mathcal{G}^{(+)}_\theta$ (\ref{Incl}). Moreover, if $h$ is an element of $\Nc(\Gamma)$, then, viewed as an element of $\mathcal{G}_{\theta}^{(+)}$, $h$ is of the form $[\mu,l,s]$ with $\mu \in {}^I[\OO_K^{\times, >}]$, $l \in I(1-u_\theta)^{-1}$ and $s \in \C$.
\end{lemma}

\begin{proof}
There is a surjective morphism from $\Aff(\HH \times \C)$ to the group
$$G=
\left\{
\left[
\begin{pmatrix}
\mu_1 & 0\\
0 & \mu_2
\end{pmatrix},
\begin{pmatrix}
l_1 \\
l_2
\end{pmatrix}
\right] \in \text{GL}(2, \C) \ltimes \C^2\vert \mu_1\in \R_>;\  \mu_2\in \C^*;\  l_1, l_2 \in \R \right\}$$
sending
$$\left[
\begin{pmatrix}
\mu & 0\\
\lambda & \nu
\end{pmatrix},
\begin{pmatrix}
l \\
s
\end{pmatrix}
\right] \to
\left[
\begin{pmatrix}
\mu & 0\\
0 & \frac{\nu}{\mu}
\end{pmatrix},
\begin{pmatrix}
l \\
\frac{\lambda}{\mu}
\end{pmatrix}
\right]
$$
The kernel is $T_\C$, so $\Aff(\HH \times \C)/T_\C$ is isomorphic to $G$. The group $\Gamma/(\Gamma \cap T_\C)$ is isomorphic to $\langle u_\theta \rangle \ltimes I$. Let $\hat{g_i}$ denote the class of the generator $g_i$ in the quotient $\Gamma / (\Gamma \cap T_\C)$. Now let $h \in \Nc(\Gamma)$ with 
$$\hat{h} =
\left[
\begin{pmatrix}
\mu_1 & 0\\
0 & \mu_2
\end{pmatrix},
\begin{pmatrix}
l_1 \\
l_2
\end{pmatrix}
\right]
$$
the image of $h$ in $G$. Then $\hat{h}$ is contained in the normalizer of $\Gamma/(\Gamma \cap T_\C)$ in $G$. Which, in particular, means that $\hat{h}\hat{g_0}\hat{h}^{-1} \in \Gamma/(\Gamma\cap T_\C)$, so 
$$\left[
\begin{pmatrix}
\sigma_1(u_\theta) & 0\\
0 & \sigma_2(u_\theta)
\end{pmatrix},
\begin{pmatrix}
l_1\sigma_1(1-u_\theta) \\
l_2\sigma_2(1-u_\theta)
\end{pmatrix}
\right] \in \langle u_\theta \rangle \ltimes I
$$
Thus there exists $l \in K$ such that $\sigma_i(l)=l_i$, $i \in \{1,2\}$. In fact, $l$ is an element in $I(1-u_\theta)^{-1}$.
Similarly, we have $\hat{h}\hat{g_i}\hat{h}^{-1} \in \Gamma/(\Gamma\cap T_\C)$, so
$$
\left[
I_2,
\begin{pmatrix}
\mu_1\sigma_1(x_i) \\
\mu_2\sigma_2(x_i)
\end{pmatrix}
\right] \in \langle u_\theta\rangle \ltimes I,\text{ for }i\in\{1,2\}$$
hence, once again, there exists some element $\mu \in K$ such that $\sigma_i(\mu)=\mu_i$. Moreover, $\mu$ is a unit in $\OO_K$, $\sigma_1(\mu)>0$ and $I$ a $\Z[\mu]$-fractional ideal, thus $\mu \in {}^I[\OO_K^{\times,>}]$. We have shown that $h = [\mu, l, s]\in \mathcal{G}_{\theta}^{(+)}$, where $\mu \in {}^I[\OO_K^{\times,>}]$, $l \in I(1-u_\theta)^{-1}$ and $s$ is some complex number.
\end{proof}

\begin{remark}
Proposition \ref{bound} gives an upper bound on the cardinality of the group $Q$, namely 
$$\#Q\leq [ {}^I[\OO_K^{\times,>}]: \langle u_\theta\rangle]\cdot[I(1-u_\theta)^{-1}: I] $$
\end{remark}

\noindent
The index of $I$ in $I(1-u_\theta)^{-1}$ is $\Norm(1-u_\theta)$. If we denote by $u$ the generator of ${}^I[\OO_K^{\times,>}]$ (chosen such that $\sigma_1(u)>1$), then, we can rewrite the inequality above as
\begin{align}\label{upb}
\#Q \leq \frac{\text{log}(\sigma_1(u_\theta))}{\text{log}(\sigma_1(u))}\Norm(1-u_\theta)
\end{align}

\begin{remark}
The inclusion of proposition \ref{bound} may be strict, as in example \ref{ex1} of the final section of the chapter. In contrast, example \ref{ex4} shows that equality between $Q$ and $\mathcal{H}_{\theta, [I]}$ may be achieved. 
\end{remark}

\begin{remark}
The inclusion ${}^I[\OO_K^{\times,>}]\leq \OO_K^{\times, >}$ may be strict. To see this, take for example, $K$ to be a real quadratic field with the property that $\OO_K = \Z[\eta]$, where $\eta$ is the fundamental unit of $K$. This happens, for instance, when the discriminant $D= 2,3,5,10,13 ...$ Then let $u_\theta=\eta^n$ for some positive integer $n \geq 2$. Remark 3.3 of Conrad's notes \cite{Con} shows that if $I$ is an invertible fractional ideal of the nonmaximal order $\Z[u_\theta]\subset\OO_K$, then it cannot be an ideal of the whole ring of integers $\OO_K$. Thus, for such a choice of $I$, the fundamental unit will not be contained in ${}^I[\OO_K^{\times,>}]$.
\end{remark}

\noindent
The following proposition gives the two conditions for an element of $\mathcal{H}_{\theta, [I]}$ to be contained in $Q$.

\begin{proposition}\label{equ}
Let $h \in \mathcal{H}_{\theta, [I]}$, where $h$ is represented by $[v,y]$  with  $v \in {}^I[\OO_K^{\times,>}]$ and $y\in I(1-u_\theta)^{-1}$. Let $(1-u_\theta)y= ax_1 + bx_2$ and let $M = (m_{ij}) \in \text{GL}(2,\Z)$ be the matrix associated to the action of $v$ on $I$, with respect to the $\Z$-basis $(x_1,x_2)$. Then, $h \in Q$ if and only if the following two conditions are satisfied:
\begin{flalign*}
1)&\ \ \ (v-1)e+y - \frac{m_{21}m_{22}vx_1-m_{12}m_{11}vx_2}{2} \in Ir^{-1}&&\\
2)&\ \ \ (\Norm(v)-1)t + \chi((u_\theta-1)y, e-\frac{y}{2}) +\frac{\chi(ax_1,bx_2)}{2} \in \frac{\chi(I,I)}{r}&&
\end{flalign*}
\end{proposition}

\begin{proof}
Let $s \in \C$, the condition that $[v,y]$ represents an element in $Q$ is equivalent to $[v,y,s] \in \mathcal{N}(\Gamma)$. First, assume that $[v,y,s]$ is in the normalizer. In particular, we have that  $[v,y,s]g_i[v,y,s]^{-1}\in \Gamma$, for $i \in \{0,1,2\}$. Then 
$$[v,y,s]g_0[v,y,s]^{-1}g_0^{-1}=\left[1, (1-u_\theta)y, (\Norm(v)-1)t +\frac{\chi(y,u_\theta y)}{2}\right] \in \langle g_1,g_2,g_3 \rangle$$
Since $(1-u_\theta)y = ax_1 + bx_2$, we compute
\begin{align*}
g_1^ag_2^b &= \left[1, ax_1+bx_2, \chi(ax_1+bx_2, e)-\frac{\chi(ax_1,bx_2)}{2}\right]\\
&=\left[1,(1-u_\theta)y, \chi((1-u_\theta)y,e)-\frac{\chi(ax_1,bx_2)}{2}\right]
\end{align*}
Hence $[v,y,s]g_0[v,y,s]^{-1}g_0^{-1}(g_1^ag_2^b)^{-1} =$
\begin{align*}
=\left[1,0,(\Norm(v)-1)t + \chi((u_\theta-1)y, e-\frac{y}{2}) +\frac{\chi(ax_1,bx_2)}{2}\right]
\end{align*}
which is central and thus contained in $\langle g_3 \rangle$. This implies condition 2) of the proposition. For $i \in \{1,2\}$, we have
\begin{align*}
[v,y,s]g_i[v,y,s]^{-1}&=\left[1,vx_i,\Norm(v)\chi(x_i,e)-\frac{\chi(y,vx_i) + \chi(y+vx_i,-y)}{2}\right]\\
&=\left[1,vx_i, \chi(vx_i,ve+y)\right] \in \langle g_1,g_2,g_3 \rangle
\end{align*}
Since $vx_i = m_{1i}x_1+m_{2i}x_2$, we have
\begin{align*}
[v,y,s]g_i[v,y,s]^{-1}(g_1^{m_{i1}}g_2^{m_{i2}})^{-1}&= \left[1,0,\chi(vx_i, (v-1)e+y)+\frac{\chi(m_{i1}x_1,m_{i2}x_2)}{2}\right]
\end{align*}
which is again central and thus contained in $\langle g_3 \rangle$. It follows that 
$$\chi(vx_1, (v-1)e+y + \frac{m_{11}m_{12}vx_2}{2}) \in \frac{\chi(I,I)}{r}$$
and
$$\chi(vx_2,(v-1)e+y-\frac{m_{21}m_{22}vx_1}{2})\in \frac{\chi(I,I)}{r}$$
and so condition 1) is satisfied. Conversely, we have to show that the images under conjugation by $[v,y,s]$ of $g_0$, $g_1$, $g_2$ and $g_3$ generate $\Gamma$. Condition 1) implies that $g_0$ is sent to $g_3^{p_0}g_1^ag_2^bg_0$ and condition 2) implies that $g_i$ is sent to $g_3^{p_i}g_1^{m_{i1}}g_2^{m_{i2}}$ (for $i\in \{1,2\}$), where $p_j$ are integers. Depending on the norm of $v$, $g_3$ is sent to either $g_3$ or $g_3^{-1}$. It is straightforward to verify that these conjugates generate $\Gamma$. This completes the proof.
\end{proof}

\begin{remark}
If the parameter $r$ is even, then the two conditions of proposition \ref{equ} take the simpler form: 
\begin{flalign*}
1)&\ \ \ (v-1)e+y \in Ir^{-1}&&\\
2)&\ \ \ (\Norm(v)-1)t + \chi((u_\theta-1)y, e-\frac{y}{2}) \in \frac{\chi(I,I)}{r}&&
\end{flalign*}
\end{remark}

\noindent
Since the case where $r$ is odd is more complicated, it is, in practice, easier to compute the automorphism group of the surface $S^{(+)}_{\theta,2r,x_1,x_2,e;t}$, which is a $[2:1]$ quotient of $S^{(+)}_{\theta,r,x_1,x_2,e;t}$, and then decide which automorphisms lift to the covering. We will now describe the algorithm for finding the group $Q$ when one is given the parameters $(\theta,r,x_1,x_2,e;t)$ (as in definition \ref{gr}).\par

\noindent
{\bf Step 1}: Determine the fundamental unit $\eta$ of $\OO_K$ (such that $\sigma_1(\eta)>1$). Let $M_1$ be the matrix with rational entries that represents the action of $\eta$ on $I$ with respect to the $\Z$-basis $(x_1,x_2)$. That is:
$$M_1
\begin{pmatrix}
x_1 \\
x_2
\end{pmatrix}= \eta \begin{pmatrix}
x_1 \\
x_2
\end{pmatrix}$$

\noindent
{\bf Step 2}: Let $j$ be the smallest positive integer such that $M_1^j$ is an integer matrix. We know that $j\leq [\langle u_\theta\rangle:\langle\eta\rangle]$, so its value can be determined by initially setting $j=1$ and incrementing it until the integrality condition is satisfied.  Then $u \coloneq \eta^j$ is the generator of ${}^I[\OO_K^{\times,>}]$.\par

\noindent
{\bf Step 3}: Find a complete system of representatives $L$ for the quotient $\frac{I(1-u_\theta)^{-1}}{I}$.\par

\noindent
{\bf Step 4}:  Let $n \coloneq [\langle u_\theta\rangle:\langle u\rangle]$. Denote by $\widetilde{v}$ the class of $v \in {}^I[\OO_K^{\times,>}]$ in $\frac{{}^I[\OO_K^{\times,>}]}{\langle u_\theta\rangle}$ and by $\widehat{y}$ the class of $y\in I(1-u_\theta)^{-1}$ in $\frac{I(1-u_\theta)^{-1}}{I}$. We now have that 
$$\mathcal{H}_{\theta,[I]} = \left\{\widetilde{u}^i\vert i\in\left\{0\dots n-1\right\}\right\} \ltimes \left\{\widehat{y}\vert y \in L\right\}$$
and by proposition (\ref{equ}),
$$Q=\left\{[\widetilde{v},\widehat{y}]\in \mathcal{H}_{\theta,[I]}\ \vert\  [v,y]\text{ satisfies conditions 1) and 2)}\right\}$$
Since $\mathcal{H}_{\theta,[I]}$ is finite, we may check its elements one by one against the two conditions.

\subsection{Computed examples}\label{expls} We will now use the method developed in the previous section to compute the group $Q$ in some particular cases. Before we begin, we will make two remarks that will help with the ease of notation.

\begin{remark}
Let $K$ be a real quadratic number field and $L_1 = {}_\Z\langle x_1, x_2\rangle$ and $L_2= {}_\Z\langle y_1,y_2\rangle$ two lattices in $K$. We will denote by $[L_1 : L_2]$ the rational number $\vert\frac{\chi(y_1,y_2)}{\chi(x_1,x_2)}\vert$. If $L_2 \leq L_1$, this is, of course, the usual group index.
\end{remark}

\begin{remark}
When writing the elements of groups $\mathcal{H}_{\theta,[I]}$ and $Q$ explicitly we will often identify an element of the field $K$ with its image through the embedding $\sigma_1$.
\end{remark}

\noindent
We begin with an example where the parameter $r$ is even and both $e \in K$ and $t \in \C$ are set to zero. Under these assumptions, conditions 1) and 2) of proposition \ref{equ} simplify significantly.

\begin{example} \label{ex1}
Let $\theta = 6$, $r = 6$, $e = 0$ and $t = 0$. Let $x_1 = 1$ and $x_2 = \eta $, where $\eta$ is the fundamental unit of $\OO_K$ with $\sigma_1(\eta)>1$. Thus, the fractional ideal is $I = {}_\Z\langle 1, \eta\rangle$ and the Inoue surface associated to these parameters is $X = S^{(+)}_{6, 6, 1, u, 0; 0}$.
\end{example}

\noindent
{\bf Step 1}: $\sigma_1(u_\theta) = 3 + 2\sqrt{2}$; $\sigma_1(\eta) = 1+\sqrt{2}$. The minimal polynomial of $\eta$ is $X^2-2X-1$. Thus 
$$
\begin{pmatrix}
0 & 1\\
1 & 2
\end{pmatrix}
\begin{pmatrix}
1 \\
\eta
\end{pmatrix}=
\eta\begin{pmatrix}
1 \\
\eta
\end{pmatrix}
$$
and so
$$
M_1=\begin{pmatrix}
0 & 1\\
1 & 2
\end{pmatrix}
$$

\noindent
{\bf Step 2}: The matrix $M_1$ is integral, so $j=1$ and $u=\eta$ is the generator of ${}^I[\OO_K^{\times,>}]$.\par

\noindent
{\bf Step 3}: $I=\OO_K\simeq \Z[\sqrt{2}]$ and $\Z[u_\theta]\simeq\Z[2\sqrt{2}]$. Then 
$$I(1-u_\theta) \simeq \Z[\sqrt{2}](2+2\sqrt{2}) = {}_\Z\langle 2,2\sqrt{2}\rangle$$
and thus a complete system of representatives for $\frac{I(1-u_\theta)^{-1}}{I}$ is
$$L = \left\{0,\frac{1}{2},\frac{\sqrt{2}}{2},\frac{1+\sqrt{2}}{2}\right\}$$

\noindent
{\bf Step 4}: $n=2$ and so
$$\mathcal{H}_{\theta,[I]}= \left\{\widetilde{1}, \widetilde{1+\sqrt{2}}\right\}\ltimes \left\{\widehat{0},\widehat{\frac{1}{2}},\widehat{\frac{\sqrt{2}}{2}},\widehat{\frac{1+\sqrt{2}}{2}}\right\}$$
Let $[v,y] \in {}^I[\OO_K^{\times,>}] \ltimes I(1-u_\theta)^{-1}$ representing an element in $Q$. In this example, the two conditions of proposition \ref{equ} become
\begin{flalign*}
1)&\ \ \ y \in I6^{-1}&&\\
2)&\ \ \ \frac{1}{2}\chi(u_\theta y,y) \in \frac{\chi(I,I)}{6}&&
\end{flalign*}
Both conditions are independent of $v$, so $v$ may take any value in ${}^I[\OO_K^{\times,>}]$. The second condition is equivalent to $\Norm(y) \in \frac{2}{6}[\Z[u_\theta]: I]\Z$. Since $[\Z[u_\theta]: I]=\frac{1}{2}$, we get that $\Norm(y)\in \frac{1}{6}\Z$. This implies that $y$ is congruent to either $0$ or $\frac{\sqrt{2}}{2}$ mod $I$ and therefore
$$Q = \left\{\widetilde{1},\widetilde{1+\sqrt{2}}\right\}\ltimes\left\{\widehat{0},\widehat{\frac{\sqrt{2}}{2}}\right\}$$
Since the action of $u = 1+\sqrt{2}$ on $\{\widehat{0},\widehat{\sqrt{2}/2}\}$ is trivial, we conclude that, in this case, the group $Q$ is isomorphic to $(\Z/2\Z)\times (\Z/2\Z)$. In this case the inclusion given by proposition \ref{bound} is strict.\par

\noindent
For the following two examples we will consider situations where $r$ is even and $[I]=1$ in the ideal class monoid $\text{Cl}(\Z[u_\theta])$. Before we address the examples, we have the following proposition.

\begin{proposition} \label{excep}
Let $\theta \geq 3$ and let $I$ be a principal fractional ideal of the order $\Z[u_\theta]$. The only case where ${}^I[\OO_K^{\times,>}]$ is strictly larger than $\langle u_\theta \rangle$ happens when $\theta = 3$.
\end{proposition}

\begin{proof}
The fact that $I$ is principal means that, up to ideal class, we may assume that $I = \Z[u_\theta]$. Let $u$ denote the generator of ${}^I[\OO_K^{\times,>}]$ (with $\sigma_1(u)>1$), then $u_\theta = u^n$, for some integer $n$, so $\Z[u_\theta]\leq \Z[u]$. But this means that $\Z[u_\theta]$ is an integral $\Z[u]$-ideal and, since $1 \in \Z[u_\theta]$, the only possibility is $\Z[u]=\Z[u_\theta]$. We thus have that $\pm u = u_\theta + b$ for some $b \in \Z$. The norm of $u_\theta + b$ is $\Norm(u_\theta + b)= (u_\theta + b)(\varphi(u_\theta)+b)= 1 + b^2 + b\theta$. If $1+b^2+b\theta = 1$, then $b=0$ or $b= -\theta$, which implies that $u \in \{\pm u_\theta, \pm \varphi(u_\theta)\}$. We know that $u \in \OO_K^{\times, >}$, so $u$ can only be $u_\theta$ or $\varphi(u_\theta)= \frac{1}{u_\theta}$. So, in this case, ${}^I[\OO_K^{\times, >}]= \langle u_\theta\rangle$. If $1+b^2+b\theta=-1$, then $b(b+\theta)= -2$, so $\theta = 3$ and $b$ is either $-1$ or $-2$. We have that $\sigma_1(u_\theta)=\frac{3+\sqrt{5}}{2}$ and since $\sigma_1(u)>1$, we see that $u$ is the exceptional unit $u_\theta -1$, so $\sigma_1(u)=\frac{1+\sqrt{5}}{2}$. We have shown that this situation, where $\theta = 3$, is the only case where the inclusion $\langle u_\theta \rangle \leq {}^I[\OO_K^{\times, >}]$ is strict. \end{proof}

\begin{example}\label{ex2} Let $\theta = 4$, $r = 6$, $x_1 = 1$ and $ x_2 = u_\theta$ (so $I = \Z[u_\theta]$). Let $t=0$ and $e= \frac{1}{6(1-u_\theta)}$.
\end{example}
\noindent
It is not difficult to check that the group defined by these parameters is in standard form. Applying proposition \ref{excep} to this example, we expect the generator of ${}^I[\OO_K^{\times, >}]$ to be $u_\theta$.\par

\noindent
{\bf Step 1}: $\sigma_1(u_\theta)=2+\sqrt{3}$. The fundamental unit $\eta$ is equal to $u_\theta$ and 
$$M_1=
\begin{pmatrix}
0 & 1\\
-1 & 4
\end{pmatrix}$$

\noindent
{\bf Step 2}: The matrix $M_1$ is integral, so $u=\eta=u_\theta$ is the generator of ${}^I[\OO_K^{\times, >}]$.\par

\noindent
{\bf Step 3}: $I = \Z[u_\theta]\simeq \Z[\sqrt{3}]$. Then 
$$I(1-u_\theta)\simeq \Z[\sqrt{3}](1+\sqrt{3})={}_\Z\langle1+\sqrt{3},2\rangle$$
and thus a complete system of representatives for $\frac{I(1-u_\theta)^{-1}}{I}$ is
$$L = \left\{0, \frac{1+\sqrt{3}}{2}\right\}$$

\noindent
{\bf Step 4}: $n=1$ and so 
$$\mathcal{H}_{\theta,[I]}= \left\{\widetilde{1} \right\}\ltimes \left\{\widehat{0}, \widehat{\frac{1+\sqrt{3}}{2}}\right\}$$
Let $[v,y] \in {}^I[\OO_K^{\times,>}] \ltimes I(1-u_\theta)^{-1}$ representing an element in $Q$. We have that $v\in \langle u_\theta \rangle$, and thus $(v-1)e \in Ir^{-1}$, so the two conditions of proposition \ref{equ} are
\begin{flalign*}
1)&\ \ \ y \in I6^{-1}&&\\
2)&\ \ \ \chi((u_\theta -1)y, \frac{1}{6(1-u_\theta)}-\frac{y}{2}) \in \frac{\chi(I,I)}{6}&&
\end{flalign*}
We only have to check if $y = \frac{u_\theta-1}{2}$ ($\sigma_1(\frac{u_\theta-1}{2})=\frac{1+\sqrt{3}}{2}$) satisfies condition 2), which it does, as one can easily verify. We get that
$$Q= \left\{\widetilde{1} \right\}\ltimes \left\{\widehat{0}, \widehat{\frac{1+\sqrt{3}}{2}}\right\}$$
so in this case, the group $Q$ is isomorphic to $\Z/2\Z$.

\begin{example}\label{ex3}
As in the previous example, let $\theta =  4$, $r = 6$, $x_1 = 1$, $x_2 = u_\theta$ and $t = 0$, but for this example take $e = 0$.
\end{example}

\noindent 
Up until the last step everything is the same as in the previous example \ref{ex2}, so again we have
$$\mathcal{H}_{\theta,[I]}= \left\{\widetilde{1} \right\}\ltimes \left\{\widehat{0}, \widehat{\frac{1+\sqrt{3}}{2}}\right\}$$
Let $[v,y] \in {}^I[\OO_K^{\times,>}] \ltimes I(1-u_\theta)^{-1}$ representing an element in $Q$. This time, the two conditions of proposition \ref{equ} are
\begin{flalign*}
1)&\ \ \ y \in I6^{-1}&&\\
2)&\ \ \ \frac{1}{2}\chi(u_\theta y,y) \in \frac{\chi(I,I)}{6}&&
\end{flalign*}
Condition 2) is equivalent to $\Norm(y)\in \frac{1}{3}[\Z[u_\theta]:I]\Z=\frac{1}{3}\Z$. The norm of $y=\frac{1+\sqrt{3}}{2}$ is $-\frac{1}{2}$, so in this case, the group $Q$ is trivial.\par

\noindent
The last example of the section is one where the group $Q$ is noncommutative.

\begin{example}\label{ex4}
Let $\theta = 7$, $r = 10$, $e=0$ and $t=0$. Let $x_1=1$ and $x_2=\eta$, where $\eta$ is the fundamental unit of $\OO_K$ with $\sigma_1(\eta)>1$. The fractional ideal is $I = \Z[\eta]$.
\end{example}

\noindent
{\bf Step 1}: $\sigma_1(u_\theta) = \frac{7+3\sqrt{5}}{2}$; $\sigma_1(\eta) =\frac{1+\sqrt{5}}{2}$. The minimal polynomial of $\eta$ is $X^2-X-1$. Thus 
$$
\begin{pmatrix}
0 & 1\\
1 & 1
\end{pmatrix}
\begin{pmatrix}
1 \\
\eta
\end{pmatrix}=
\eta\begin{pmatrix}
1 \\
\eta
\end{pmatrix}
$$
and so
$$
M_1=\begin{pmatrix}
0 & 1\\
1 & 1
\end{pmatrix}
$$

\noindent
{\bf Step 2}: The matrix $M_1$ is integral, so $j=1$ and $u=\eta$ is the generator of ${}^I[\OO_K^{\times,>}]$.\par

\noindent
{\bf Step 3}: $I=\OO_K=\Z[u]$ and $\Z[u_\theta]=\Z[u^4]$. Then
$$I(1-u_\theta)=\Z[u](3u+1)={}_\Z\langle u+2,5\rangle$$
and thus a complete system of representatives for $\frac{I(1-u_\theta)^{-1}}{I}$ is
$$L = \left\{k\frac{1+3u}{5}\ \vert\ k\in \{0\dots 4\}\right\}$$

\noindent
{\bf Step 4}: $n=4$ and so 
$$\mathcal{H}_{\theta,[I]}=\{\widetilde{u}^i\ \vert\ i\in\{0\dots 3\}\}\ltimes \left\{k\widehat{\frac{1+3u}{5}}\ \vert\ k\in \{0\dots 4\}\right\}$$
Let $[v,y] \in {}^I[\OO_K^{\times,>}] \ltimes I(1-u_\theta)^{-1}$ representing an element in $Q$. The two conditions of proposition \ref{equ} are
\begin{flalign*}
1)&\ \ \ y \in I10^{-1}&&\\
2)&\ \ \ \frac{1}{2}\chi(u_\theta y,y) \in \frac{\chi(I,I)}{10}&&
\end{flalign*}
One then checks that all the elements of $\mathcal{H}_{\theta,[I]}$ satisfy the conditions. Therefore
$$Q=\{\widetilde{u}^i\ \vert\ i\in\{0\dots 3\}\}\ltimes \left\{k\widehat{\frac{1+3u}{5}}\ \vert\ k\in \{0\dots 4\}\right\}$$
so $Q$ is isomorphic to $(\Z/4\Z) \ltimes (\Z/5\Z)$, where the action of the generator of $(\Z/4\Z)$ on $(\Z/5\Z)$ is multiplication by $3$.

\section{Surfaces of type $S^{(-)}$} \label{s-}
In this chapter we will describe the automorphism groups of Inoue surfaces of type $S^{(-)}$. The results are similar to those of the previous chapter, so we will not prove them, instead we will only point the differences. We begin with the analogous construction.

Let $\theta$ be an integer, $\theta \geq 1$ and $g_\theta$ the polynomial 
$$g_{\theta}(X)=X^2 - \theta X - 1$$
Let $K = \Q[X]/(g_\theta)$ and $u_\theta:= \hat{X} \in K$. The group $\mathcal{G}^{(-)}_\theta$ is defined in the exact same way as $\mathcal{G}_{\theta}^{(+)}$ was defined in \ref{Incl}, also its action on $\HH \times \C$ remains the same. 
\begin{definition}
Similar to definition \ref{gr}, we consider the subgroup $\Gamma^{(-)}_{\theta,r,x_1,x_1,e} \leq \mathcal{G}^{(-)}_\theta$ generated by
$$g_0=[u_\theta,0,0];\ \ g_i=[1,x_i,\chi(x_i, e)], i \in \{1,2\};\ \ g_3=\left[1,0,-\frac{\chi(x_1,x_2)}{r}\right]$$
We say that $\Gamma^{(-)}_{\theta,r,x_1,x_1,e}$ is in standard form if $\langle g_3 \rangle  = \Gamma_{\theta, r, x_1,x_2, e}^{(-)} \cap T_{\C}$ (see \ref{trans}).
\end{definition}

\begin{proposition}
The quotient of $\HH \times \C$ by the left action of the group $\Gamma^{(-)}_{\theta,r,x_1,x_1,e}$ is an Inoue surface of type $S^{(-)}$. Moreover, any surface of type $S^{(-)}$ is biholomorphic to one obtained in this way.
\end{proposition}

\begin{definition}
We will denote by $S^{(-)}_{\theta,r,x_1,x_2,e}$ the Inoue surface obtained as the quotient of $\HH \times \C$ by the left action of $\Gamma_{\theta,r,x_1,x_2,e}^{(-)}$.   
\end{definition}
\noindent
Let $X$ be the be the Inoue surface $S^{(-)}_{\theta,r,x_1,x_2,e}$ and $\Gamma = \Gamma^{(-)}_{\theta, r, x_1,x_2,e}$. Let $I= {}_\Z\langle x_1,x_2\rangle$, the fractional ideal of $\Z[u_\theta]$. The automorphism group of $X$ is again $\mathcal{N}(\Gamma)/\Gamma$, where $\mathcal{N}(\Gamma)$ is the normailzer of $\Gamma$ in $\Aff(\HH \times \C)$. As in the case discussed in the previous chapter, we have that $\mathcal{N}(\Gamma) \leq \mathcal{G}_{\theta}^{(-)}$ and an element $h \in \mathcal{N}(\Gamma)$ is of the form $h= [v,y,s]$ with $v \in {}^I[\OO_K^{\times,>}]$, $y\in I(1-u_\theta)^{-1}$ and $s\in \C$.

\begin{proposition} Let $[v,y,s] \in \mathcal{G}^{(-)}_{\theta}$ with $v \in {}^I[\OO_K^{\times,>}]$, $y\in I(1-u_\theta)^{-1}$ and $s\in \C$. Let $(1-u_\theta)y = ax_1+bx_2$ and let the action of $v$ on $I$ (with respect to the $\Z$-basis $(x_1,x_2)$) be given by the matrix $M=(m_{ij})$. Then, the condition that $[v,y,s]$ is contained in the normalizer $\mathcal{N}(\Gamma)$ is equivalent to the following two statements:
\begin{flalign*}
1)&\ \ \ (v-1)e+y - \frac{m_{21}m_{22}vx_1-m_{12}m_{11}vx_2}{2} \in Ir^{-1}&&\\
2)&\ \ \ 2s + \chi((u_\theta-1)y, e-\frac{y}{2}) +\frac{\chi(ax_1,bx_2)}{2} \in \frac{\chi(I,I)}{r}&&
\end{flalign*}
\end{proposition}
\noindent
This implies that the automorphism group of $X$ is finite and is an extension of the form
$$\begin{tikzcd}[ampersand replacement=\&,cramped]
	1 \& {\Z/2\Z} \& {\text{Aut}(X)} \& Q \& 1
	\arrow[from=1-1, to=1-2]
	\arrow[from=1-2, to=1-3]
	\arrow[from=1-3, to=1-4]
	\arrow[from=1-4, to=1-5]
\end{tikzcd}$$
where $Q$ is again subgroup in 
$$\mathcal{H}_{\theta,[I]} = \frac{{}^I[\OO_K^{\times,>}]}{\langle u_\theta \rangle} \ltimes \frac{I(1-u_\theta)^{-1}}{I}$$
The computation of the group $Q$ for particular choices of initial data is done by similar methods to the $S^{(+)}$ case.


\begin{thebibliography}{1000}

\bibitem[Bog]{Bog}
\newblock F. A. Bogomolov,
\newblock {\em Classification of Surfaces of Class $VII_{0}$ with $b_2= 0$},
\newblock Math. USSR Izvestija, Vol. 10, No. 2, (1976), 255--269.

\bibitem[BoMo]{BM}
\newblock  S. Bochner, D. Montgomery, 
\newblock {\em Groups On Analytic Manifolds.}
\newblock Annals of Mathematics, vol. 48, No. 3, (1947), 659--69.

\bibitem[Con]{Con}
\newblock K. Conrad,
\newblock {\em The conductor of an ideal},
\newblock available at\\https://kconrad.math.uconn.edu/blurbs/gradnumthy/conductor.pdf

\bibitem[FuNa]{FN}
\newblock Y. Fujimoto, N. Nakayama,
\newblock {\em Compact complex surfaces admitting non-trivial surjective endomorphisms},
\newblock Tohoku Mathematical Journal, Tohoku Math. J. (2) 57(3), (2005), 395--426.

\bibitem[In]{Ino}
\newblock M. Inoue, 
\newblock {\em On surfaces of class ${\rm VII}_0$}, 
\newblock Invent. Math. {\bf 24},  (1974), 269--310.

\bibitem[KhTel]{KT}
\newblock Z. Khaled, A. Teleman,
\newblock {\em On the classification of Inoue surfaces},
\newblock arXiv: 2406.15158

\bibitem[Nam]{Nam}
\newblock M. Namba,
\newblock {\em Automorphism groups of Hopf surfaces},
\newblock Tohoku Mathematical Journal, Tohoku Math. J. (2) 26(1),(1974), 133--157.

\bibitem[OeMi]{OeMi}
\newblock  K. Oeljeklaus, C. Miebach,
\newblock {\em Compact complex non-{K}\"{a}hler manifolds associated with totally real reciprocal units},
\newblock Math. Zeitschrift {\bf 301},  (2022),  2747--2760.

\bibitem[PrSh]{PrSh}
\newblock Y. Prokhorov, C. Shramov,
\newblock {\em  Automorphism groups of Inoue and Kodaira surfaces},
\newblock Asian J. Math, Vol. 24, No. 2, (2020), 355--368.

\bibitem[Tel]{Tel}
\newblock A. Teleman,
\newblock {\em Projectively flat surfaces and Bogomolov's theorem on class $\text{VII}_0$ surfaces},
\newblock International Journal of Mathematics, 5, No. 2, (1994), 253--264.

\bibitem[Wa]{Wall}
\newblock  C.T.C. Wall,
\newblock {\em Geometric structures on compact complex analytic manifolds},
\newblock Topology {\bf 25}, (1986), 119--153.

\end{thebibliography}
\end{document}